\pdfoutput=1

\newif\ifpersonal
\personaltrue

\documentclass[12pt,a4paper,dvipsnames,x11names]{amsart}

\usepackage[utf8]{inputenc}
\usepackage[T1]{fontenc}
\usepackage[english]{babel}
\usepackage[a4paper,margin=1in]{geometry}

\usepackage{microtype,lmodern}
\usepackage{mathpazo}
\usepackage{euler}

\usepackage{amsmath,amssymb,amsthm,mathtools,mathrsfs,bm,eucal,tensor}
\usepackage{stmaryrd}
\usepackage{dsfont}

\usepackage[all,cmtip]{xy}
\usepackage{tikz}
\usetikzlibrary{arrows.meta,calc,positioning,fit,backgrounds,shapes.geometric}
\usepackage{tikz-cd}

\tikzset{
  >=Stealth,
  box/.style={draw, rounded corners, inner sep=6pt},
  v/.style={draw, circle, inner sep=1.4pt},
  arr/.style={->, thick},
  darr/.style={->, thick, dashed}
}

\usepackage{enumerate,comment,braket,xspace,csquotes}
\usepackage{xcolor}
\usepackage{hyperref}
\hypersetup{
  colorlinks=true,
  linkcolor=blue!60!black,
  citecolor=blue!60!black,
  urlcolor=blue!60!black
}
\usepackage[capitalize,nameinlink]{cleveref}

\numberwithin{equation}{section}

\theoremstyle{plain}
\newtheorem{theorem}{Theorem}[section]
\newtheorem{proposition}[theorem]{Proposition}
\newtheorem{lemma}[theorem]{Lemma}

\newtheorem{example}[theorem]{Example}

\theoremstyle{definition}
\newtheorem{definition}[theorem]{Definition}

\theoremstyle{remark}
\newtheorem{remark}[theorem]{Remark}

\newcommand{\CC}{\mathbb C}

\newcommand{\II}{\mathscr I}
\newcommand{\OO}{\mathscr O}

\newcommand{\End}{\mathrm{End}}

\newcommand{\GL}{\mathrm{GL}}

\newcommand{\reg}{\mathrm{reg}}
\newcommand{\sing}{\mathrm{sing}}

\begin{document}

\title[Log Bott localization with non-isolated lci zero varieties]{Log Bott localization with non-isolated lci zero varieties}

\author{Maur\'icio Corr\^ea}
\address{Maur\'icio Corr\^ea\\
Universit\`a degli Studi di Bari\\
Via E. Orabona 4, I-70125 Bari, Italy}
\email[M. Corr\^ea]{mauricio.barros@uniba.it, mauriciomatufmg@gmail.com}

\author{Elaheh Shahsavaripour}
\address{Elaheh Shahsavaripour\\
Universit\`a degli Studi di Bari\\
Via E. Orabona 4, I-70125 Bari, Italy}
\email[Elaheh Shahsavaripour]{e.shahsavaripour@gmail.com}

\date{}

\begin{abstract}
We establish a logarithmic Bott localization formula for global holomorphic sections of $T_X(-\log D)$ on a compact complex manifold $X$ with simple normal crossings divisor $D$. The zero scheme is allowed to have non-isolated compact components, assumed to be local complete intersections and to satisfy the natural Bott nondegeneracy condition. We further develop a current-theoretic formulation and, in the local complete intersection case, identify the local residue term with a Coleff--Herrera current.
\end{abstract}

\maketitle

\section{Introduction}\label{sec:intro}

The localization of characteristic classes at the singular set of a vector field is a classical theme in complex geometry. In the holomorphic setting, Bott's residue formula shows that characteristic numbers of a compact complex manifold can be recovered from local data attached to the zero set of a holomorphic vector field. A crucial point is that Bott's theorem is not restricted to isolated zeros: it already allows higher-dimensional zero components, provided they are smooth and non-degenerate in the first-order sense determined by the induced action on the normal bundle \cite{Bott67b}. Baum and Bott later extended this perspective from vector fields to holomorphic foliations, thereby establishing the standard residual interpretation of characteristic classes in the foliated setting \cite{BaumBott72}.

More recently, several extensions of this circle of ideas have been developed. On the logarithmic side, a Baum--Bott formula for foliations by curves, together with a singular-variety version obtained through log resolutions, was established in \cite{CorreaLourencoMachado24}. The determination of residues attached to holomorphic vector fields and foliations is a classical theme and remains one of the most delicate problems in the subject.  Recent years have witnessed substantial progress on these questions; among representative contributions in this direction, we mention \cite{CorreaSuwa25,CorreaLourenco19,KaufmannLarkangWulcan23}. For a broader perspective on the interplay among indices of vector fields, residue theories for foliations, and the geometry of singular varieties, we refer to \cite{SuwaBook,BrasseletSeadeSuwa09,CorreaSeade26}.  

The purpose of the present paper is to prove a logarithmic Bott-type localization theorem for global logarithmic vector fields with non-isolated zero sets. We work on a compact complex manifold $X$ endowed with a simple normal crossings divisor $D$, and consider a logarithmic vector field
$$
v\in H^0\bigl(X,T_X(-\log D)\bigr).
$$
We assume that the zero scheme of $v$ is a finite disjoint union
$$
Z(v)=\bigsqcup_{j=1}^m Z_j
$$
of compact, reduced, pure-dimensional local complete intersection subspaces. Along each component $Z_j$ we consider the Bott action induced by $v$ on the conormal sheaf, and we impose the natural nondegeneracy condition that this action be invertible. Under these hypotheses, the characteristic number
$$
\int_X \Phi\bigl(T_X(-\log D)\bigr)
$$
localizes on the components of $Z(v)$.
The resulting formula is a logarithmic counterpart of Bott's non-isolated residue theorem in a more singular geometric setting. The novelty is not merely that the zero set may have positive-dimensional components. Rather, the ambient geometry is logarithmic, the relevant bundle is $T_X(-\log D)$, and the zero components are allowed to be local complete intersections, hence need not be smooth a priori. The local residue is therefore described not only by transverse linearization, but by the intrinsic Bott action on the conormal bundle together with the induced decomposition of the logarithmic tangent bundle along the smooth locus of each component. 
A concrete illustration is provided by the Fulton--MacPherson space $(\mathbf P^2)[2]$, namely the compactification of the moduli space of two ordered distinct points in $\mathbf P^2$ \cite{FultonMacPherson94}. In this case the boundary divisor is the exceptional divisor of the blow-up of $\mathbf P^2\times \mathbf P^2$ along the diagonal, and a natural one-parameter subgroup of $\operatorname{PGL}_3$ induces a global logarithmic vector field whose zero locus has positive-dimensional components both outside and inside the boundary. Our theorem then yields the explicit identity
$$
\int_{(\mathbf P^2)[2]} c_4\bigl(T_{(\mathbf P^2)[2]}(-\log D)\bigr)=6,
$$
recovering this characteristic number by localization on a genuine moduli space with non-isolated logarithmic zero components; see example ~\ref{subsec:fm-p2-two-points}.

Moreover, the logarithmic localization theorem also applies after passage to a functorial log resolution \cite{CorreaLourencoMachado24}. More precisely, if
$
\pi:(X,D)\to Y
$
is a log resolution of a compact normal variety and the transformed object is generated on $X$ by a global logarithmic vector field
$
\widetilde v\in H^0\bigl(X,T_X(-\log D)\bigr),
$
then, under the hypotheses of the main theorem, namely that the zero scheme of $\widetilde v$ is a finite union of compact reduced Bott non-degenerate local complete intersection components, the associated logarithmic residues are well defined and compute the global characteristic number by localization, see example \ref{exe}.

The non-isolated situation is not merely a technical generalization. Even on compact complex manifolds carrying many global vector fields or distributions, positive-dimensional singular sets arise naturally. For instance, on Hopf manifolds, singular holomorphic distributions of dimension or codimension one cannot be purely isolated in dimension at least three: either the singular set is empty or it has a positive-dimensional component. This follows from Baum--Bott type arguments and was proved in \cite{CorreaFerreiraFernandezPerez16}. It is therefore natural to seek localization formulas whose local terms are supported not only at isolated points but also along compact higher-dimensional components.

The main theorem extends Bott's non-isolated localization principle from smooth non-degenerate zero manifolds to compact reduced lci zero varieties in the logarithmic setting. At the same time, it complements the logarithmic Baum--Bott theory of \cite{CorreaLourencoMachado24}: that work is formulated for foliations and singular varieties via log resolutions, whereas the present article works directly on a smooth logarithmic pair and treats a global logarithmic vector field as the basic object.

The paper is organized as follows. Section~2 constructs the Bott action on the conormal bundle of each zero component and formulates the corresponding nondegeneracy condition. Section~3 recalls the current-theoretic tools needed later and proves the explicit Coleff--Herrera realization of the local contribution. Section~4 studies tubular neighborhoods of compact lci subspaces. Section~5 contains the transgression argument and the proof of the main theorem.

Let $X$ be a compact complex manifold of complex dimension $n$, let $D\subset X$ be a simple normal crossings divisor, and consider the logarithmic tangent bundle $T_X(-\log D)$. A section
$$
v\in H^0\bigl(X,T_X(-\log D)\bigr)
$$
is, by definition, a logarithmic derivation. In local SNC coordinates
$D=\{z_1\cdots z_\ell=0\}$ it has the form
\begin{equation}\label{eq:local-log-v-intro}
v=\sum_{i=1}^{\ell} a_i(z)\,z_i\frac{\partial}{\partial z_i}+\sum_{j=\ell+1}^{n} a_j(z)\,\frac{\partial}{\partial z_j},
\qquad a_i\in \OO_X,
\end{equation}
see Proposition \ref{prop:log-local} below.
The scheme-theoretic zero locus $Z(v)\subset X$ is defined by the ideal sheaf
\begin{equation}\label{eq:zero-ideal}
\II_{Z(v)}:=Im\Big(\Omega_X^1(\log D)\xrightarrow{ i_v }\OO_X\Big),
\end{equation}
where $i_v$ is contraction. Equivalently, in any local trivialization
$T_X(-\log D)|_U\simeq \OO_U^{\oplus n}$, if $v$ is represented by $(v_1,\dots,v_n)$ then $\II_{Z(v)}|_U=(v_1,\dots,v_n)$. In particular, the underlying set of $Z(v)$ is $\{p\in X\mid v(p)=0\}$.

Now, assume that $Z(v)$ is a finite disjoint union of compact, reduced, pure-dimensional lci analytic subspaces
\begin{equation}\label{eq:Zdecomp}
Z(v)=\bigsqcup_{j=1}^m Z_j,
\end{equation}
where each $Z_j$ has pure complex dimension $r_j$ and codimension $k_j=n-r_j$. Write $\II_j:=\II_{Z_j}$.

Because $Z_j$ is reduced and $v$ vanishes pointwise on $Z_j$, the derivation $v$ preserves $\II_j$ and hence $\II_j^2$. Therefore $v$ induces an $\OO_{Z_j}$-linear endomorphism of the conormal sheaf
\begin{equation}\label{eq:Astar-intro}
A_{Z_j}^*:\II_j/\II_j^2\longrightarrow \II_j/\II_j^2,
\qquad [f]\longmapsto [v(f)].
\end{equation}
Since $Z_j$ is lci, $\II_j/\II_j^2$ is locally free of rank $k_j$, and the normal bundle is
\begin{equation}\label{eq:normal-intro}
N_{Z_j/X}:= \mathscr{H}om_{\OO_{Z_j}}(\II_j/\II_j^2,\OO_{Z_j}).
\end{equation}
Dualizing \eqref{eq:Astar-intro} gives the normal endomorphism
\begin{equation}\label{eq:A-intro}
A_{Z_j}:=(A_{Z_j}^*)^\vee:N_{Z_j/X}\longrightarrow N_{Z_j/X}.
\end{equation}

\begin{definition}[Bott-nondegeneracy condition]\label{def:nondeg-intro}
We say that $v$ is \emph{Bott nondegenerate along $Z_j$} if $A_{Z_j}^*$, equivalently $A_{Z_j}$, is pointwise invertible on $Z_j$.
\end{definition}

At smooth points of $Z_j$, $A_{Z_j}^*$ agrees with the transverse Jacobian of $v$ in normal coordinates; see Proposition \ref{prop:jacobian}.

For the localization formula one needs, on $Z_j $, the quotient of the logarithmic tangent bundle by the normal directions determined by $Z_j$. Define
\begin{equation}\label{eq:rho-map-intro}
\rho_j:T_X(-\log D)|_{Z_j }\longrightarrow N_{Z_j/X}|_{Z_j },
\qquad
\xi\longmapsto \big([f]\mapsto [\xi(f)]\big).
\end{equation}
By Lemma \ref{lem:rho-well-defined}, this is a well-defined holomorphic bundle morphism.

\begin{definition} \label{def:surj-intro}
We say that the zero component $Z_j$ is \emph{logarithmically transversal} if the map $\rho_j$ is surjective on $Z_j $.
\end{definition}

Under this hypothesis the kernel
\begin{equation}\label{eq:Kj-intro}
K_j:=\ker(\rho_j)\subset T_X(-\log D)|_{Z_j }
\end{equation}
is a holomorphic vector bundle of rank $r_j$, and one has a short exact sequence
\begin{equation}\label{eq:exact-KN-intro}
0\longrightarrow K_j\longrightarrow T_X(-\log D)|_{Z_j }\xrightarrow{\ \rho_j\ }N_{Z_j/X}|_{Z_j }\longrightarrow 0.
\end{equation}
In the classical situation $D=\varnothing$ and $Z_j$ smooth, or more generally whenever $Z_j $ is a log-smooth submanifold of $(X,D)$, the bundle $K_j$ identifies with the tangent bundle of the regular locus.

Fix a $\GL_n(\CC)$-invariant polynomial $\Phi$ on $\mathfrak{gl}_n(\CC)$, homogeneous of degree $n$.
We write
\begin{equation}\label{eq:CW-notation-intro}
\int_X \Phi\bigl(T_X(-\log D)\bigr)
\end{equation}
for the characteristic number determined by the invariant polynomial $\Phi$ and the logarithmic tangent bundle $T_X(-\log D)$.
On $Z_j $ choose Hermitian metrics on $K_j$ and on $N_{Z_j/X}|_{Z_j }$, with Chern connections $\nabla_{K,j}$ and $\nabla_{N,j}$ and curvatures $\Omega_{K,j}$ and $\Omega_{N,j}$.
If $v$ is Bott nondegenerate along $Z_j$, then $A_{Z_j}$ is invertible on $Z_j $, so the inhomogeneous form $\det(A_{Z_j}+\Omega_{N,j})$ has invertible degree-$0$ term and therefore admits an inverse in the graded algebra of differential forms; see Lemma \ref{lem:invdet}.

\begin{definition} \label{def:resform-intro}
Assume Bott nondegeneracy and the surjectivity hypothesis along $Z_j$. Define
\begin{equation}\label{eq:Resform-intro}
\operatorname{Res}_{Z_j}(\Phi)
:=
\Big(\Phi(A_{Z_j}+\Omega_{K,j})\wedge \det(A_{Z_j}+\Omega_{N,j})^{-1}\Big)_{[2r_j]}
\in \Omega^{2r_j}(Z_j^\reg).
\end{equation}
Here $(-)_{[2r_j]}$ denotes the component of total degree $2r_j$.
\end{definition}

To integrate over a possibly singular component we use the fundamental current. For a compact reduced analytic subspace $Z\subset X$ of pure dimension $r$, define
\begin{equation}\label{eq:frown-intro}
\langle [Z],\eta\rangle:=\int_{Z}\eta|_{Z },
\qquad \eta\in \Omega_c^{2r}(X).
\end{equation}
Lemma \ref{lem:fundamental-current} recalls that $[Z]$ is a well-defined closed current.

We can now state the localization theorem in its final intrinsic form.

\begin{theorem}\label{thm:main}
In the standing setup above, assume that for every irreducible component $Z_j$ of the zero scheme of $v$ the following hold:
\begin{enumerate} 
\item $Z_j$ is compact, reduced, pure-dimensional, and lci;
\item $v$ is Bott nondegenerate along $Z_j$ in the sense of Definition \ref{def:nondeg-intro};
\item the natural morphism
$$
\rho_j:T_X(-\log D)|_{Z_j}\to N_{Z_j/X}|_{Z_j}
$$
is surjective.
\end{enumerate}
Then, for every $\GL_n(\CC)$-invariant polynomial $\Phi$ of degree $n$,
\begin{equation}\label{eq:main}
\int_X \Phi\bigl(T_X(-\log D)\bigr)
=
\sum_{j=1}^m \langle [Z_j],\operatorname{Res}_{Z_j}(\Phi)\rangle,
\end{equation}
where
$$
\operatorname{Res}_{Z_j}(\Phi)
=
\Bigl(\Phi(A_{Z_j}+\Omega_{K,j})\wedge \det(A_{Z_j}+\Omega_{N,j})^{-1}\Bigr)_{[2r_j]}
\in \Omega^{2r_j}(Z_j^{\reg}).
$$
\end{theorem}

In the reduced lci case, the local contribution
$$
\langle [Z_j],\operatorname{Res}_{Z_j}(\Phi)\rangle
$$
admits an explicit Coleff--Herrera realization. More precisely, let $(U_{j,\alpha})_\alpha$ be an open cover of a neighborhood of $Z_j$ such that on each $U_{j,\alpha}$ the ideal sheaf $\mathcal I_{Z_j}$ is generated by a regular sequence
$
f_{j,\alpha}=(f_{j,\alpha,1},\dots,f_{j,\alpha,k_j}),
$
and set
$$
R_{j,\alpha}
:=
\bar\partial(1/f_{j,\alpha,1})\wedge\cdots\wedge \bar\partial(1/f_{j,\alpha,k_j}),
\qquad
\omega_{j,\alpha}
:=
df_{j,\alpha,1}\wedge\cdots\wedge df_{j,\alpha,k_j}.
$$
Choose a partition of unity $(\chi_{j,\alpha})_\alpha$ subordinate to $(U_{j,\alpha})_\alpha$, and choose smooth forms
$$
\eta_{j,\alpha}\in \mathscr A^{2r_j}(U_{j,\alpha})
$$
such that
$$
\eta_{j,\alpha}|_{Z_j^{\reg}\cap U_{j,\alpha}}
=
\operatorname{Res}_{Z_j}(\Phi).
$$
Then
\begin{equation}\label{eq:main-CH}
\langle [Z_j],\operatorname{Res}_{Z_j}(\Phi)\rangle
=
\frac{1}{(2\pi i)^{k_j}}
\sum_\alpha
\left\langle
R_{j,\alpha},
\chi_{j,\alpha}\,\eta_{j,\alpha}\wedge \omega_{j,\alpha}
\right\rangle.
\end{equation}
This expression is independent of the chosen cover, of the local generators, of the local representatives $\eta_{j,\alpha}$, and of the partition of unity. If $Z_j$ is smooth, then
\begin{equation}\label{eq:main-smooth}
\langle [Z_j],\operatorname{Res}_{Z_j}(\Phi)\rangle
=
\int_{Z_j}\operatorname{Res}_{Z_j}(\Phi).
\end{equation}
 
If $D=\varnothing$ and each $Z_j$ is smooth, then the map $\rho_j$ is the usual quotient map $T_X|_{Z_j}\to N_{Z_j/X}$, so $K_j\simeq T Z_j$ and \eqref{eq:main} reduces to Bott's non-isolated zero formula \cite{Bott67a,Bott67b,BottTu,SuwaBook}. If each $Z_j$ is a point, then \eqref{eq:main} becomes the classical isolated-zero residue $\Phi(A_p)/\det(A_p)$.

\subsection*{Acknowledgments}
We are very grateful to Francesco Bastianelli and Fernando Louren\c co for numerous insightful conversations.  MC is a member of INdAM--GNSAGA.

\section{Logarithmic preliminaries and the conormal Bott action}\label{sec:logtangent}

Let $X$ be a complex manifold and $D\subset X$ a simple normal crossings divisor. If $p\in X$ and if $(z_1,\dots,z_n)$ are local holomorphic coordinates centered at $p$ such that
\begin{equation}\label{eq:snc-local}
D=\{z_1\cdots z_\ell=0\}
\end{equation}
in a neighborhood of $p$, then the sheaf $T_X(-\log D)$ is locally free with frame
\begin{equation}\label{eq:log-frame}
z_1\partial_{z_1},\dots,z_\ell\partial_{z_\ell},\partial_{z_{\ell+1}},\dots,\partial_{z_n}.
\end{equation}
Consequently every local section of $T_X(-\log D)$ has the form \eqref{eq:local-log-v-intro}. 
By definition, $T_X(-\log D)$ is the dual of $\Omega_X^1(\log D)$, and in SNC coordinates the latter is freely generated by
\begin{equation}\label{eq:log-forms-local}
\frac{dz_1}{z_1},\dots,\frac{dz_\ell}{z_\ell},dz_{\ell+1},\dots,dz_n,
\end{equation}
see \cite[\S 2]{Saito80} or \cite[\S 6]{Deligne70}. The dual frame is exactly \eqref{eq:log-frame}.

The next lemma is the basic algebraic input behind the conormal action.

\begin{lemma}\label{lem:ideal-preserved}
Let $Z\subset X$ be a reduced analytic subspace and let $v\in H^0(X,T_X(-\log D))$ be such that $v(p)=0$ for all $p\in Z$. Then
\begin{equation}\label{eq:ideal-preserved}
v(\II_Z)\subset \II_Z
\qquad \text{and} \qquad
v(\II_Z^2)\subset \II_Z^2.
\end{equation}
\end{lemma}

\begin{proof}
Let $f\in \II_Z$. Since $f|_Z=0$, for every $p\in Z$ one has
\begin{equation}\label{eq:derivative-zero}
(v(f))(p)=df_p(v(p))=0.
\end{equation}
Because $Z$ is reduced, the ideal sheaf $\II_Z$ is precisely the sheaf of holomorphic functions vanishing on $Z$; hence $v(f)\in \II_Z$. The inclusion $v(\II_Z^2)\subset \II_Z^2$ follows from the Leibniz rule. Indeed, if $g=\sum_\nu a_\nu b_\nu$ with $a_\nu,b_\nu\in \II_Z$, then
\begin{equation}\label{eq:leibniz-I2}
v(g)=\sum_\nu \big(v(a_\nu)b_\nu+a_\nu v(b_\nu)\big)\in \II_Z^2.
\end{equation}
\end{proof}

\begin{definition}\label{def:Astar}
Under the hypotheses of Lemma \ref{lem:ideal-preserved}, define the conormal Bott endomorphism
\begin{equation}\label{eq:Astar}
A_Z^*:\II_Z/\II_Z^2\longrightarrow \II_Z/\II_Z^2,
\qquad [f]\longmapsto [v(f)].
\end{equation}
\end{definition}

\begin{lemma}\label{lem:Astar-linear}
The map $A_Z^*$ is well defined and $\OO_Z$-linear.
\end{lemma}

\begin{proof}
Lemma \ref{lem:ideal-preserved} shows that $v(\II_Z)\subset \II_Z$ and $v(\II_Z^2)\subset \II_Z^2$, so \eqref{eq:Astar} is well defined on the quotient.
Let $a\in \OO_X$ and $f\in \II_Z$. Then
\begin{equation}\label{eq:OZ-linearity-Astar}
v(af)=v(a)f+a\,v(f).
\end{equation}
Since $f\in \II_Z$, the first term lies in $\II_Z^2$, hence $[v(af)]=a|_Z\,[v(f)]$ in $\II_Z/\II_Z^2$. This is precisely $\OO_Z$-linearity.
\end{proof}

\begin{remark}[Normal bundle in the lci case]\label{rem:lci-conormal}
If $Z\hookrightarrow X$ is lci of codimension $k$, then $\II_Z/\II_Z^2$ is locally free of rank $k$ and its dual
\begin{equation}\label{eq:normal-lci}
N_{Z/X}:=  \mathscr{H}om_{\OO_Z}(\II_Z/\II_Z^2,\OO_Z)
\end{equation}
is the normal bundle; see \cite[\S 3]{Fischer76}.
\end{remark}

\begin{proposition}\label{prop:jacobian}
Assume that $Z\subset X$ is lci and let $p\in Z $. Choose holomorphic coordinates
\begin{equation}\label{eq:coords-jac}
(x_1,\dots,x_r,y_1,\dots,y_k)
\end{equation}
centered at $p$ such that $Z^\reg=\{y_1=\cdots=y_k=0\}$ near $p$. Write
\begin{equation}\label{eq:v-coordinates}
v=\sum_{a=1}^r u_a(x,y)\frac{\partial}{\partial x_a}+\sum_{b=1}^k w_b(x,y)\frac{\partial}{\partial y_b}.
\end{equation}
Then the matrix of $A_Z^*$ in the conormal basis $[y_1],\dots,[y_k]$ of $\II_Z/\II_Z^2$ at $p$ is
\begin{equation}\label{eq:jacobian}
\left(\frac{\partial w_b}{\partial y_c}(p)\right)_{1\le b,c\le k}.
\end{equation}
Equivalently, the matrix of the dual endomorphism $A_Z$ on $N_{Z/X,p}$ is the transpose of \eqref{eq:jacobian}.
\end{proposition}

\begin{proof}
Because $v$ vanishes along $Z$, each $w_b$ belongs to the ideal $(y_1,\dots,y_k)$. Modulo $\II_Z^2=(y_1,\dots,y_k)^2$, one can write
\begin{equation}\label{eq:w-linear-part}
w_b(x,y)\equiv \sum_{c=1}^k \frac{\partial w_b}{\partial y_c}(x,0)\,y_c \pmod{\II_Z^2}.
\end{equation}
Since $v(y_b)=w_b$, the class $A_Z^*([y_b])=[v(y_b)]$ is represented by this linear part. Evaluating at $p$ gives the matrix \eqref{eq:jacobian}. Dualization yields the statement for $A_Z$.
\end{proof}

\begin{remark}[Logarithmic eigenvectors along the boundary]\label{rem:log-eigen}
In an SNC chart
\begin{equation}\label{eq:snc-eigen-chart}
D=\{z_1\cdots z_\ell=0\},
\qquad
v=\sum_{i\le \ell} a_i(z)\,z_i\partial_{z_i}+\sum_{i>\ell} a_i(z)\,\partial_{z_i},
\end{equation}
one has $v(z_i)=a_i(z)z_i$ for $i\le \ell$. If $z_i\in \II_Z$, then $[z_i]\in \II_Z/\II_Z^2$ is an eigenvector of $A_Z^*$ with eigenvalue $a_i|_Z$. Thus $A_Z^*$ records the transverse logarithmic weights.
\end{remark}

We next justify the bundle morphism used in the definition of $K_j$.

\begin{lemma}\label{lem:rho-well-defined}
Let $Z\subset X$ be lci. On $Z$ the assignment
\begin{equation}\label{eq:rho-map}
\rho:T_X(-\log D)|_{Z}\longrightarrow N_{Z/X}|_{Z},
\qquad
\xi\longmapsto \big([f]\mapsto [\xi(f)]\big)
\end{equation}
is a well-defined holomorphic morphism of vector bundles.
\end{lemma}

\begin{proof}
Fix $p\in Z$. For a local section $\xi$ of $T_X(-\log D)$ and $f\in \II_Z$, the derivative $\xi(f)|_Z$ depends only on the class $[f]\in \II_Z/\II_Z^2$. Indeed, if $f'\equiv f \mod \II_Z^2$, then $f'-f\in \II_Z^2$, hence by the Leibniz rule $\xi(f'-f)\in \II_Z$. Therefore $[\xi(f')]=[\xi(f)]$ in $\OO_Z$.
The resulting map $\II_Z/\II_Z^2\to \OO_Z$ is $\OO_Z$-linear: if $a\in \OO_X$, then
\begin{equation}\label{eq:rho-linearity-proof}
\xi(af)=\xi(a)f+a\,\xi(f),
\end{equation}
and the first term vanishes on $Z$. Thus $\rho(\xi)$ is a section of $N_{Z/X}=  \mathscr{H}om(\II_Z/\II_Z^2,\OO_Z)$. Holomorphicity follows because, in local frames for $T_X(-\log D)$ and $N_{Z/X}$, the coefficients of $\rho$ are holomorphic derivatives of local defining equations of $Z$.
\end{proof}

\begin{lemma}\label{lem:invdet}
Let $F\to Y$ be a rank-$k$ complex vector bundle, let $A\in H^0(Y,\End(F))$ be pointwise invertible, and let $\Omega\in \Omega^2(Y,\End(F))$ be any endomorphism-valued $2$-form. Then
\begin{equation}\label{eq:invdet-statement}
\det(A+\Omega)
\in \bigoplus_{q=0}^{k}\Omega^{2q}(Y)
\end{equation}
has invertible degree-$0$ term $\det(A)$ and therefore admits a unique inverse in the graded algebra of inhomogeneous forms.
\end{lemma}

\begin{proof}
Factor
\begin{equation}\label{eq:factor-det}
\det(A+\Omega)=\det(A)\det(\mathrm{Id}+A^{-1}\Omega).
\end{equation}
The form $A^{-1}\Omega$ has strictly positive degree, hence
\begin{equation}\label{eq:nilpotent}
\eta:=\det(\mathrm{Id}+A^{-1}\Omega)-1
\end{equation}
also has strictly positive degree and is nilpotent in the finite-dimensional graded algebra $$\bigoplus_{q=0}^{\dim Y}\Omega^{2q}(Y).$$ Therefore
\begin{equation}\label{eq:geometric-series}
(1+\eta)^{-1}=1-\eta+\eta^2-\cdots+(-1)^N\eta^N
\end{equation}
for $N\gg 0$. Multiplying by $\det(A)^{-1}$ gives the required inverse.
\end{proof}

We also record the local product structure of the divisor in a flow box adapted to a nowhere vanishing logarithmic vector field.

\begin{lemma}\label{lem:product-divisor}
Let $v\in H^0(X,T_X(-\log D))$ be nowhere vanishing on an open set $U\subset X$. Then every $p\in U$ has a neighborhood $W\subset U$ biholomorphic to $V\times \Delta$ with coordinates $(w,t)$ such that
 $
v=\partial/\partial t
$
and
\begin{equation}\label{eq:product-divisor}
D\cap W=(D_V\cap V)\times \Delta
\end{equation}
for an analytic hypersurface $D_V\subset V$.
\end{lemma}

\begin{proof}
By the holomorphic flow-box theorem, after shrinking around $p$ there are coordinates $(w,t)$ on $W\simeq V\times \Delta$ such that $v=\partial_t$; see \cite[Chapter I]{SuwaBook}. Let $g(w,t)$ be a local reduced equation of $D\cap W$. Since $v$ is logarithmic along $D$, we have
\begin{equation}\label{eq:g-preserved}
\partial_t g=v(g)=h\,g
\end{equation}
for some holomorphic function $h$ on $W$. Solving this first-order linear equation in $t$ yields
\begin{equation}\label{eq:g-solution}
g(w,t)=u(w,t)\,g_0(w),
\end{equation}
where $u$ is nowhere vanishing and $g_0(w):=g(w,0)$. Thus $D\cap W=\{g_0(w)=0\}\times \Delta$.
\end{proof}

\section{Fundamental currents and Coleff--Herrera currents}\label{sec:CH}

We first recall the fundamental current of a reduced analytic cycle.

\begin{lemma}\label{lem:fundamental-current}
Let $Z\subset X$ be a compact reduced analytic subspace of pure complex dimension $r$. Then the rule
\begin{equation}\label{eq:fund-current}
\langle [Z],\eta\rangle:=\int_{Z}\eta|_{Z},
\qquad \eta\in \Omega_c^{2r}(X),
\end{equation}
defines a well-defined closed current of degree $2r$ supported on $Z$. Moreover, if $\beta$ is a smooth top-degree form on $Z$, then $\beta$ is locally integrable near $Z^\sing$ and
\begin{equation}\label{eq:pairing-fund-current}
\langle [Z],\beta\rangle=\int_{Z}\beta
\end{equation}
is well defined.
\end{lemma}

\begin{proof}
The current of integration along a reduced pure-dimensional analytic set is classical; see \cite{King71,HarveyShiffman74}. Because $Z^\sing$ has complex codimension at least $1$ in $Z$, it has real codimension at least $2$ in the real manifold $Z\cup Z^\sing$ of top dimension $2r$. Hence any smooth $2r$-form on $Z$ is locally bounded with respect to a smooth ambient volume form and is therefore locally integrable across $Z^\sing$; see \cite[\S III.3]{Demailly}. Closedness of $[Z]$ is part of the same standard result.
\end{proof}

We next recall the Coleff--Herrera current associated with an lci component.

Throughout this section, $Z\subset X$ is a compact reduced lci analytic subspace of codimension $k$.
Let $U\subset X$ be Stein and let $f=(f_1,\dots,f_k)$ be a regular sequence in $\OO_X(U)$ generating $\II_Z|_U$.
For a compactly supported test form $\varphi\in \Omega_c^{n,n-k}(U)$ define the local Coleff--Herrera product by the polytube limit
\begin{equation}\label{eq:CH-def}
\Big\langle \bar\partial\Big(\frac{1}{f_1}\Big)\wedge\cdots\wedge \bar\partial\Big(\frac{1}{f_k}\Big),\ \varphi\Big\rangle
:=
\lim_{\varepsilon\to 0}\int_{T_\varepsilon}\frac{\varphi}{f_1\cdots f_k},
\qquad
T_\varepsilon:=\{|f_1|=\varepsilon_1,\dots,|f_k|=\varepsilon_k\}.
\end{equation}
The determinant transformation law on overlaps, and hence the gluing, are classical; see \cite{CH78,Passare88} and \cite[\S III]{Demailly}.

\begin{proposition}\label{prop:RZ}\cite{CH78,Passare88,Demailly}.
There exists a canonically defined global current
\begin{equation}\label{eq:RZ}
R_Z\in \mathscr D'_Z\bigl(X,\det(N_{Z/X})^\vee\bigr)
\end{equation}
such that locally, when $\II_Z$ is generated by a regular sequence $f=(f_1,\dots,f_k)$, one has
\begin{equation}\label{eq:RZ-local}
R_Z|_U=\bar\partial\Big(\frac{1}{f_1}\Big)\wedge\cdots\wedge \bar\partial\Big(\frac{1}{f_k}\Big)
\end{equation}
in the induced trivialization of $\det(N_{Z/X})^\vee$.
\end{proposition}

\begin{theorem}[Coleff--Herrera--Poincar\'e--Lelong]\cite{CH78,Passare88}  \cite[\S III]{Demailly}\label{thm:CH-PL}
Let $Z\subset X$ be reduced and lci of codimension $k$. Then the scalar current $R_Z\wedge df$ is globally well defined and equals $(2\pi i)^k[Z]$ near $Z$. Equivalently, for every compactly supported top-degree form $\eta$ near $Z$,
\begin{equation}\label{eq:CH-equals-Z}
\langle [Z],\eta\rangle
=
\frac{1}{(2\pi i)^k}\,\big\langle R_Z,\eta\wedge df_1\wedge\cdots\wedge df_k\big\rangle
=
\frac{1}{(2\pi i)^k}\,\lim_{\varepsilon\to 0}\int_{T_\varepsilon}
\eta\wedge \frac{df_1}{f_1}\wedge\cdots\wedge\frac{df_k}{f_k}.
\end{equation}
\end{theorem}

\section{Tubes around compact lci components}\label{sec:tubes}

For the global localization proof we need tubular neighborhoods shrinking to a compact lci component. The following statement is sufficient for the argument and avoids any unnecessary quantitative estimate.

\begin{proposition}\label{prop:tubes}
Let $X$ be a complex manifold of complex dimension $n$, and let $Z\subset X$ be a compact reduced lci analytic subspace of pure dimension $r$ and codimension $k=n-r\ge 1$. Then there exist:
\begin{enumerate} 
\item an open neighborhood $W$ of $Z$ and a smooth function $\varphi:W\to [0,\infty)$ such that $\varphi^{-1}(0)=Z$;
\item a number $\varepsilon_0>0$ such that for every regular value $\varepsilon\in (0,\varepsilon_0)$ the set
\begin{equation}\label{eq:Ueps-def}
U_\varepsilon(Z):=\{x\in W\mid \varphi(x)<\varepsilon\}
\end{equation}
is a neighborhood of $Z$ with smooth boundary $\partial U_\varepsilon(Z)=\varphi^{-1}(\varepsilon)$;
\item on $Z$, after choosing a Hermitian metric on $N_{Z/X}|_{Z}$, the restriction of $\partial U_\varepsilon(Z)$ over $Z$ is canonically diffeomorphic to the radius-$\sqrt{\varepsilon}$ sphere bundle of $N_{Z/X}|_{Z}$;
\item for every smooth $2n$-form $\beta$ on $X$ one has
\begin{equation}\label{eq:beta-tube-goes-to-0-prop}
\lim_{\varepsilon\to 0}\int_{U_\varepsilon(Z)}\beta=0.
\end{equation}
Consequently, if
\begin{equation}\label{eq:Meps-local}
M_\varepsilon:=X\setminus \mathrm{int}(U_\varepsilon(Z)),
\end{equation}
then $$\lim_{\varepsilon\to 0}\int_{M_\varepsilon}\beta=\int_X\beta$$
\end{enumerate}
\end{proposition}

\begin{proof}
Because $Z$ is lci of codimension $k$, every point of $Z$ has a neighborhood $U_\alpha$ on which $\II_Z$ is generated by a regular sequence $f_\alpha=(f_{\alpha,1},\dots,f_{\alpha,k})$. After passing to a finite cover of the compact set $Z$, choose a partition of unity $\{\rho_\alpha\}$ subordinate to these neighborhoods and set
\begin{equation}\label{eq:phi-global}
\varphi:=\sum_{\alpha}\rho_\alpha\sum_{j=1}^k |f_{\alpha,j}|^2.
\end{equation}
Then $\varphi\ge 0$ and $\varphi^{-1}(0)=Z$.
By Sard's theorem, the set of critical values of the smooth function $\varphi$ has measure zero, so every sufficiently small regular value $\varepsilon$ gives a smooth boundary $\partial U_\varepsilon(Z)$. Since the sublevel sets shrink to $Z$, they form a neighborhood basis of $Z$.
On $Z^\reg$ the inclusion $Z^\reg\hookrightarrow X$ is a smooth embedded submanifold of real codimension $2k$. The smooth tubular neighborhood theorem identifies a neighborhood of the zero section in the real normal bundle with a neighborhood of $Z$ in $X$. Because $Z$ is lci, the real normal bundle underlying $N_{Z/X}|_{Z}$ is precisely this smooth normal bundle, and the level sets of $\varphi$ are small normal spheres after shrinking. This proves the third assertion. 
Finally, let $dV$ be any smooth positive volume form on $X$. Analytic subsets of positive codimension have zero $2n$-dimensional measure in $X$; see \cite{King71,Demailly}. Since $U_\varepsilon(Z)$ decreases to $Z$ as $\varepsilon\to 0$, continuity from above for the finite measure determined by $dV$ gives
\begin{equation}\label{eq:volume-to-zero}
\lim_{\varepsilon\to 0}\int_{U_\varepsilon(Z)} dV=\int_Z dV=0.
\end{equation}
If $\beta$ is any smooth $2n$-form, there exists $C>0$ such that $|\beta|\le C dV$ on the compact manifold $X$, hence
\begin{equation}\label{eq:beta-bound-by-volume}
\Big|\int_{U_\varepsilon(Z)}\beta\Big|\le C\int_{U_\varepsilon(Z)}dV\xrightarrow[\varepsilon\to 0]{}0.
\end{equation}
The last statement follows from
\begin{equation}\label{eq:split-X}
\int_{M_\varepsilon}\beta=\int_X\beta-\int_{U_\varepsilon(Z)}\beta.
\end{equation}
\end{proof}

\section{Proof of Theorem \ref{thm:main}}\label{sec:proof-final}

Fix a $\GL_n(\CC)$-invariant polynomial $\Phi$ on $\mathfrak{gl}_n(\CC)$, homogeneous of degree $n$. Choose a Hermitian metric on $T_X(-\log D)$ and let $\nabla^h$ be its Chern connection, with curvature $\Omega^h$. By the standard Chern--Weil theorem,
\begin{equation}\label{eq:CW-proof}
\int_X \Phi\bigl(T_X(-\log D)\bigr)=\int_X \Phi(\Omega^h).
\end{equation}
We shall therefore work with the differential form $\Phi(\Omega^h)$ and recover at the end the intrinsic notation of \eqref{eq:main}.

Let $Z(v)=\bigsqcup_{j=1}^m Z_j$ be the standing decomposition into compact reduced lci components. For each $j$, choose a tube $U_{j,\varepsilon}:=U_\varepsilon(Z_j)$ as in Proposition \ref{prop:tubes}. For $\varepsilon>0$ sufficiently small these tubes are pairwise disjoint. Set
\begin{equation}\label{eq:Meps}
M_\varepsilon:=X\setminus \mathrm{int}\Big(\bigsqcup_{j=1}^m U_{j,\varepsilon}\Big),
\qquad
\partial M_\varepsilon=\bigsqcup_{j=1}^m \partial U_{j,\varepsilon}.
\end{equation}
By Proposition \ref{prop:tubes},
\begin{equation}\label{eq:limit}
\int_X \Phi\bigl(T_X(-\log D)\bigr)=\lim_{\varepsilon\to 0}\int_{M_\varepsilon}\Phi(\Omega^h).
\end{equation}

Set $U:=X\setminus Z(v)$. On $U$ the logarithmic holomorphic vector field
$$
v\in H^0\bigl(X,T_X(-\log D)\bigr)
$$
is nowhere zero. By the holomorphic flow-box theorem and Lemma \ref{lem:product-divisor}, $U$ admits a locally finite cover by holomorphic flow boxes
\begin{equation}\label{eq:flow-box-cover}
W_\alpha\simeq V_\alpha\times \Delta_\alpha
\end{equation}
with coordinates $(w,t)$ in which
\begin{equation}\label{eq:flow-box-v}
v=\frac{\partial}{\partial t}
\qquad \text{and} \qquad
D\cap W_\alpha=(D_\alpha\cap V_\alpha)\times \Delta_\alpha.
\end{equation}
Consequently all local constructions may be carried out on $T_X(-\log D)|_{W_\alpha}$ as on the pullback of a logarithmic bundle from the base $V_\alpha$.

On each $W_\alpha$ choose a smooth connection $\nabla^\alpha$ on $T_X(-\log D)|_{W_\alpha}$ which is \emph{basic} in Bott's sense, namely whose connection matrix has no $dt$ and no $d\overline t$ component in a product trivialization. Concretely, choose a smooth frame on the slice $V_\alpha\times\{0\}$ and extend it constantly along $t$ and $\overline t$, declaring the $t$ and $\overline t$ covariant derivatives to vanish. Next choose a smooth partition of unity $\{\rho_\alpha\}$ subordinate to $\{W_\alpha\}$ such that each $\rho_\alpha$ depends only on the $w$ variables. Define the global Bott connection on $T_X(-\log D)|_U$ by
\begin{equation}\label{eq:global-bott}
\nabla^{\mathrm B}:=\sum_\alpha \rho_\alpha\,\nabla^\alpha.
\end{equation}
Because the differentials $d\rho_\alpha$ have no $dt,d\overline t$ components, the connection $\nabla^{\mathrm B}$ is again basic on each flow box. Hence its curvature $\Omega^{\mathrm B}$ is pulled back from the $(2n-2)$-real-dimensional base $V_\alpha$, and therefore
\begin{equation}\label{eq:basic-vanishing}
\Phi(\Omega^{\mathrm B})\equiv 0
\qquad \text{on } U
\end{equation}
by degree reasons. This is the usual Bott vanishing argument \cite{Bott67a,BottTu,SuwaBook}.
Consider the   transgression form between $\nabla^{\mathrm B}$ and $\nabla^h$. For
\begin{equation}\label{eq:nabla-t}
\nabla^t:=\nabla^{\mathrm B}+t(\nabla^h-\nabla^{\mathrm B})
\end{equation}
with curvature $\Omega^t$, define
\begin{equation}\label{eq:CS-def}
\Theta_\Phi(\nabla^{\mathrm B},\nabla^h)
:=
n\int_0^1 \Phi\big(\nabla^h-\nabla^{\mathrm B},\Omega^t,\dots,\Omega^t\big)\,dt
\in \Omega^{2n-1}(U).
\end{equation}
Then
\begin{equation}\label{eq:CS}
d\,\Theta_\Phi(\nabla^{\mathrm B},\nabla^h)=\Phi(\Omega^h)-\Phi(\Omega^{\mathrm B})=\Phi(\Omega^h)
\qquad \text{on } U
\end{equation}
by \cite{ChernSimons74,BottTu}. Integrating over $M_\varepsilon$ and using Stokes yields
\begin{equation}\label{eq:boundary-sum}
\int_{M_\varepsilon}\Phi(\Omega^h)
=
\sum_{j=1}^m \int_{\partial U_{j,\varepsilon}}\Theta_\Phi(\nabla^{\mathrm B},\nabla^h).
\end{equation}
Combining this with \eqref{eq:limit}, we reduce the theorem to the local boundary limits
\begin{equation}\label{eq:boundary-limit}
\int_X \Phi\bigl(T_X(-\log D)\bigr)
=
\sum_{j=1}^m \lim_{\varepsilon\to 0}\int_{\partial U_{j,\varepsilon}}\Theta_\Phi(\nabla^{\mathrm B},\nabla^h).
\end{equation}

We now fix one component $Z:=Z_j$ of dimension $r$ and codimension $k=n-r$, and prove that its boundary limit equals $\langle [Z],\operatorname{Res}_Z(\Phi)\rangle$.

Choose a relatively compact open subset $V\Subset Z$. Over $V$ the tube boundary is canonically a smooth sphere bundle
\begin{equation}\label{eq:sphere-bundle-local}
\pi_\varepsilon:S_\varepsilon(N_{Z/X}|_V)\longrightarrow V.
\end{equation}
Define the fiber integral
\begin{equation}\label{eq:alpha-eps}
\alpha_{\varepsilon,V}:=(\pi_\varepsilon)_*\Big(\Theta_\Phi(\nabla^{\mathrm B},\nabla^h)|_{S_\varepsilon(N_{Z/X}|_V)}\Big)
\in \Omega^{2r}(V).
\end{equation}
Then
\begin{equation}\label{eq:fiber-int-local}
\int_{\partial U_{\varepsilon}(Z)\cap \pi_\varepsilon^{-1}(V)}\Theta_\Phi(\nabla^{\mathrm B},\nabla^h)=\int_V \alpha_{\varepsilon,V}.
\end{equation}
Since $V$ is arbitrary and $Z^\sing$ has codimension at least one in $Z$, it is enough to identify the smooth limit of $\alpha_{\varepsilon,V}$ on compact subsets of $Z$ and then pass to the current $[Z]$ by exhaustion, using Lemma \ref{lem:fundamental-current}.

Over $V$ choose a Hermitian orthogonal splitting of the exact sequence
\begin{equation}\label{eq:split-exact}
0\longrightarrow K_Z\longrightarrow T_X(-\log D)|_V\xrightarrow{\ \rho\ }N_{Z/X}|_V\longrightarrow 0,
\end{equation}
which exists because we are in the smooth category on $V$. Using the exponential map of the chosen Hermitian metric on $N_{Z/X}|_V$, identify a neighborhood of $V$ in $X$ with a neighborhood of the zero section in $N_{Z/X}|_V$. In local coordinates $(x,y)$ adapted to $V=\{y=0\}$, Proposition \ref{prop:jacobian} gives
\begin{equation}\label{eq:linearization-normal}
v(x,y)=v_{\parallel}(x,y)+A_Z(x)y+R(x,y),
\end{equation}
where $v_{\parallel}(x,y)$ takes values in $K_Z$, the term $A_Z(x)y$ is the first-order normal part, and
\begin{equation}\label{eq:R-quadratic}
R(x,y)=O(|y|^2)
\end{equation}
in the normal variables. Bott nondegeneracy means that $A_Z(x)$ is invertible for every $x\in V$.

For $\varepsilon>0$ sufficiently small, the vector fields
\begin{equation}\label{eq:homotopy-v}
v_s(x,y):=v_{\parallel}(x,y)+A_Z(x)y+sR(x,y),
\qquad 0\le s\le 1,
\end{equation}
are nowhere zero on the sphere bundle $S_\varepsilon(N_{Z/X}|_V)$.
Indeed, the linear term $A_Z(x)y$ is uniformly bounded away from zero on the compact sphere bundle because $A_Z$ is invertible and $V$ is relatively compact, while $R(x,y)=O(|y|^2)$ becomes arbitrarily small after restricting to a sufficiently small sphere. Therefore the boundary integrals are homotopy invariant, and one may replace $v$ by its linear normal model
\begin{equation}\label{eq:linear-model}
v_{\mathrm{lin}}(x,y):=v_{\parallel}(x,0)+A_Z(x)y.
\end{equation}
This is the standard reduction step in Bott's proof; compare \cite{Bott67b,BottTu,SuwaBook}.
Keep the splitting \eqref{eq:split-exact}. On the linear model, the connection $\nabla^{\mathrm B}$ may be chosen so that, along $V$, it is block diagonal with respect to
$$
T_X(-\log D)|_V\simeq K_Z\oplus N_{Z/X}|_V,
$$
namely
\begin{equation}\label{eq:block-diagonal-connection}
\nabla^{\mathrm B}|_V=\nabla_K\oplus \nabla_N,
\end{equation}
where $\nabla_K$ and $\nabla_N$ are the Chern connections of the chosen Hermitian metrics on $K_Z$ and $N_{Z/X}|_V$.
Consequently the associated curvature matrix has the form
\begin{equation}\label{eq:block-curvature}
\Omega^{\mathrm B}|_V=
\begin{pmatrix}
\Omega_K & 0\\
0 & \Omega_N
\end{pmatrix}.
\end{equation}

After the fiberwise rescaling $y=\varepsilon u$, the sphere bundle $S_\varepsilon(N_{Z/X}|_V)$ is identified with the fixed unit sphere bundle $S_1(N_{Z/X}|_V)$. Under this identification the normal linear term becomes $A_Z(x)u$, whereas the higher-order term contributes only $O(\varepsilon)$ after rescaling, by \eqref{eq:R-quadratic}. The pullback of the transgression form therefore admits an expansion whose coefficients depend smoothly on $x\in V$ and on the angular variable $u\in S_1(N_{Z/X}|_x)$, and whose leading term is exactly the universal Bott form associated with the pair $(A_Z(x),\Omega_K(x),\Omega_N(x))$.

More concretely, only the component of total fiber degree $2k-1$ survives after integration over the sphere $S_1(N_{Z/X}|_x)\simeq S^{2k-1}$, and Bott's classical sphere integral identity applied fiberwise yields
\begin{equation}\label{eq:bott-fiber-identity}
\int_{S_1(N_{Z/X}|_x)} \lim_{\varepsilon\to 0}\Theta_\Phi(\nabla^{\mathrm B},\nabla^h)
=
\Big(\Phi(A_Z+\Omega_K)\wedge \det(A_Z+\Omega_N)^{-1}\Big)_{[2r]}(x).
\end{equation}
Equivalently,
\begin{equation}\label{eq:alpha-limit-pointwise}
\lim_{\varepsilon\to 0}\alpha_{\varepsilon,V}(x)
=
\Big(\Phi(A_Z+\Omega_K)\wedge \det(A_Z+\Omega_N)^{-1}\Big)_{[2r]}(x).
\end{equation}
Because the dependence on $x$ is smooth and the convergence is uniform on compact subsets of $V$, the convergence is smooth on $V$. By Definition \ref{def:resform-intro}, the right-hand side is exactly $\operatorname{Res}_Z(\Phi)|_V$.

Therefore, for every relatively compact $V\Subset Z$,
\begin{equation}\label{eq:local-convergence-on-V}
\lim_{\varepsilon\to 0}\int_V \alpha_{\varepsilon,V}
=
\int_V \operatorname{Res}_Z(\Phi).
\end{equation}

Choose an exhaustion $V_1\Subset V_2\Subset \cdots \Subset Z$ by relatively compact open subsets whose union is $Z$. Since $Z^\sing$ has real codimension at least $2$ in the pure $2r$-dimensional space $Z$, the form $\operatorname{Res}_Z(\Phi)$ is locally integrable near $Z^\sing$ by Lemma \ref{lem:fundamental-current}. Hence
\begin{equation}\label{eq:pairing-reg-res}
\langle [Z],\operatorname{Res}_Z(\Phi)\rangle=\int_{Z}\operatorname{Res}_Z(\Phi)=\lim_{\nu\to \infty}\int_{V_\nu}\operatorname{Res}_Z(\Phi).
\end{equation}
On each $V_\nu$ we already know, from \eqref{eq:local-convergence-on-V}, that the boundary integral converges to the residue integral. Passing first to the limit $\varepsilon\to 0$ for fixed $\nu$, and then to the exhaustion limit, gives
\begin{equation}\label{eq:component-limit}
\lim_{\varepsilon\to 0}\int_{\partial U_\varepsilon(Z)}\Theta_\Phi(\nabla^{\mathrm B},\nabla^h)
=
\langle [Z],\operatorname{Res}_Z(\Phi)\rangle.
\end{equation}
Repeating this for each component $Z_j$ and substituting into \eqref{eq:boundary-limit} yields exactly
\begin{equation}\label{eq:finish}
\int_X \Phi\bigl(T_X(-\log D)\bigr)
=
\sum_{j=1}^m \langle [Z_j],\operatorname{Res}_{Z_j}(\Phi)\rangle
=
\sum_{j=1}^m \int_{Z_j^\reg}\operatorname{Res}_{Z_j}(\Phi).
\end{equation}
This proves the equality \eqref{eq:main}.
If the Hermitian metrics and Chern connections on $K_j$ and $N_{Z_j/X}|_{Z_j }$ are replaced by different choices, the corresponding residue forms differ by an exact $2r_j$-form on $Z_j^\reg$, by a transgression on the base $Z_j$; see \cite{ChernSimons74,BottTu,Kobayashi87}. Since $[Z_j]$ is closed, pairing with $[Z_j]$ kills exact forms. Thus the sum of residues is intrinsic, as asserted in Theorem \ref{thm:main}.

\subsection*{Local Coleff--Herrera expression of the local contribution}

We now rewrite the local contribution of a reduced lci component in explicit local form. Since $Z_j$ is only assumed to be a local complete intersection, one cannot in general choose global defining equations. Accordingly, the Coleff--Herrera realization must be written on a local cover and glued by the transformation law.

Assume the hypotheses of Theorem \ref{thm:main}, and let $Z_j\subset X$ be a reduced lci component of the zero scheme of $v$, of codimension $k_j$ and dimension $r_j$. Choose an open cover $(U_{j,\alpha})_\alpha$ of a neighborhood of $Z_j$ such that on each $U_{j,\alpha}$ the ideal sheaf $\mathcal I_{Z_j}$ is generated by a regular sequence
$$
f_{j,\alpha}=(f_{j,\alpha,1},\dots,f_{j,\alpha,k_j}).
$$
On overlaps $U_{j,\alpha}\cap U_{j,\beta}$ one has
$$
f_{j,\alpha}=g_{\alpha\beta}f_{j,\beta},
$$
where $g_{\alpha\beta}$ is an invertible holomorphic matrix. The associated local Coleff--Herrera currents
\begin{equation}\label{eq:CH-local-current}
R_{j,\alpha}
:=
\bar\partial(1/f_{j,\alpha,1})\wedge\cdots\wedge\bar\partial(1/f_{j,\alpha,k_j})
\end{equation}
satisfy the transformation law
$$
R_{j,\alpha}=(\det g_{\alpha\beta})^{-1}R_{j,\beta},
$$
and the local Jacobian forms satisfy
$$
df_{j,\alpha,1}\wedge\cdots\wedge df_{j,\alpha,k_j}
=
\det(g_{\alpha\beta})\,
df_{j,\beta,1}\wedge\cdots\wedge df_{j,\beta,k_j}
\quad \text{mod } \mathcal I_{Z_j}.
$$
Hence the collection $(R_{j,\alpha})_\alpha$ defines canonically the global Coleff--Herrera current associated with $Z_j$, and one has locally
\begin{equation}\label{eq:CH-PL-local}
(2\pi i)^{k_j}[Z_j]
=
R_{j,\alpha}\wedge df_{j,\alpha,1}\wedge\cdots\wedge df_{j,\alpha,k_j}.
\end{equation}

Let
$$
\operatorname{Res}_{Z_j}(\Phi)
=
\Bigl(\Phi(A_{Z_j}+\Omega_{K,j})\wedge \det(A_{Z_j}+\Omega_{N,j})^{-1}\Bigr)_{[2r_j]}
\in \Omega^{2r_j}(Z_j^{\reg})
$$
be the residue form of Definition \ref{def:resform-intro}. On each $U_{j,\alpha}$ choose a smooth closed form
$$
\eta_{j,\alpha}\in \mathscr A^{2r_j}(U_{j,\alpha})
$$
such that
$$
\eta_{j,\alpha}|_{Z_j^{\reg}\cap U_{j,\alpha}}
=
\operatorname{Res}_{Z_j}(\Phi).
$$
Finally, let $(\chi_{j,\alpha})_\alpha$ be a smooth partition of unity subordinate to $(U_{j,\alpha})_\alpha$ on a neighborhood of $Z_j$.

\begin{proposition}\label{prop:CH-local-form}
With the notation above, one has
\begin{equation}\label{eq:CH-local-form}
\langle [Z_j],\operatorname{Res}_{Z_j}(\Phi)\rangle
=
\frac{1}{(2\pi i)^{k_j}}
\sum_\alpha
\left\langle
R_{j,\alpha},
\chi_{j,\alpha}\,\eta_{j,\alpha}\wedge
df_{j,\alpha,1}\wedge\cdots\wedge df_{j,\alpha,k_j}
\right\rangle.
\end{equation}
This expression is independent of the chosen cover, of the chosen generators $f_{j,\alpha}$, of the local representatives $\eta_{j,\alpha}$, and of the partition of unity.
\end{proposition}

\begin{proof}
By \eqref{eq:CH-PL-local}, on each $U_{j,\alpha}$ one has
$$
(2\pi i)^{k_j}[Z_j]\wedge \eta_{j,\alpha}
=
R_{j,\alpha}\wedge \eta_{j,\alpha}\wedge
df_{j,\alpha,1}\wedge\cdots\wedge df_{j,\alpha,k_j}.
$$
Multiplying by $\chi_{j,\alpha}$ and summing over $\alpha$, we obtain
$$
(2\pi i)^{k_j}\sum_\alpha \chi_{j,\alpha}[Z_j]\wedge \eta_{j,\alpha}
=
\sum_\alpha
R_{j,\alpha}\wedge \chi_{j,\alpha}\eta_{j,\alpha}\wedge
df_{j,\alpha,1}\wedge\cdots\wedge df_{j,\alpha,k_j}.
$$
Since $\sum_\alpha \chi_{j,\alpha}=1$ near $Z_j$ and each $\eta_{j,\alpha}$ restricts to $\operatorname{Res}_{Z_j}(\Phi)$ on $Z_j^{\reg}\cap U_{j,\alpha}$, the left-hand side is exactly
$$
(2\pi i)^{k_j}\,\langle [Z_j],\operatorname{Res}_{Z_j}(\Phi)\rangle.
$$
This proves \eqref{eq:CH-local-form}. The independence of the chosen generators follows from the transformation law for Coleff--Herrera currents and the corresponding transformation law for the Jacobian factor. The independence of the local representatives $\eta_{j,\alpha}$ is immediate, since only their restriction to $Z_j^{\reg}$ contributes to the pairing with $[Z_j]$.
\end{proof}

Summing \eqref{eq:CH-local-form} over all components and using Theorem \ref{thm:main}, one obtains the global current-theoretic expression
$$
\int_X \Phi\bigl(T_X(-\log D)\bigr)
=
\sum_{j=1}^m
\frac{1}{(2\pi i)^{k_j}}
\sum_\alpha
\left\langle
R_{j,\alpha},
\chi_{j,\alpha}\,\eta_{j,\alpha}\wedge
df_{j,\alpha,1}\wedge\cdots\wedge df_{j,\alpha,k_j}
\right\rangle.
$$

If $Z_j$ is smooth, one may choose a tubular retraction $\tau_j:U_j\to Z_j$ and take local representatives of the form
$$
\eta_{j,\alpha}=\tau_j^*\bigl(\operatorname{Res}_{Z_j}(\Phi)\bigr)|_{U_{j,\alpha}},
$$
so that the above formula reduces to the usual integral of the residue form over $Z_j$.
\begin{example}\rm \label{exe}
Let $Y:=\mathbb P(1,1,1,k)$, with $k\ge 2$, and let
$
\pi:(X,D)\to Y
$
be the natural resolution of the isolated quotient singularity at $[0:0:0:1]$. Thus
$$
X=\mathbb P_{\mathbb P^2}\bigl(\mathcal O_{\mathbb P^2}\oplus \mathcal O_{\mathbb P^2}(k)\bigr),
\qquad
D\simeq \mathbb P^2,
\qquad
N_{D/X}\simeq \mathcal O_{\mathbb P^2}(-k).
$$
Choose constants $a,b,c\in \mathbb C^*$ such that
\begin{equation}\label{eq:weights-example-3fold}
a\neq b,
\qquad
c\neq ka,
\qquad
c\neq kb,
\end{equation}
and consider the diagonal vector field
$$
\xi
=
a z_0\partial_{z_0}
+
a z_1\partial_{z_1}
+
b z_2\partial_{z_2}
+
c z_3\partial_{z_3}.
$$
It induces a one-dimensional foliation on $Y$, and its logarithmic pullback defines a global section
$$
\widetilde\xi\in H^0\bigl(X,T_X(-\log D)\bigr).
$$
We show that the zero scheme of $\widetilde\xi$ consists of a smooth rational curve together with one isolated non-degenerate point.
On the chart $\{z_0\neq 0\}$, with coordinates
$$
u_1=\frac{z_1}{z_0},
\qquad
u_2=\frac{z_2}{z_0},
\qquad
\tau=\frac{z_3}{z_0^k},
$$
one finds
\begin{equation}\label{eq:field-u-chart-example-3fold}
\widetilde\xi
=
(b-a)u_2\partial_{u_2}
+
(c-ka)\tau\partial_{\tau}.
\end{equation}
Hence the zero locus is given there by $u_2=\tau=0$, with $u_1$ free. The same description holds on $\{z_1\neq 0\}$, and the two pieces glue to a smooth rational curve
$$
C=\overline{\{z_2=0,\ z_3=0\}}\subset X.
$$
Geometrically, $C$ is the line $\{z_2=0\}\subset \mathbb P^2$ inside the distinguished section $\{z_3=0\}$.
On the chart $\{z_2\neq 0\}$, with coordinates
$$
v_0=\frac{z_0}{z_2},
\qquad
v_1=\frac{z_1}{z_2},
\qquad
\eta=\frac{z_3}{z_2^k},
$$
one has
$$
\widetilde\xi
=
(a-b)v_0\partial_{v_0}
+
(a-b)v_1\partial_{v_1}
+
(c-kb)\eta\partial_{\eta}.
$$
By \eqref{eq:weights-example-3fold}, this chart contains a unique isolated zero,
$$
p=\{v_0=v_1=\eta=0\}.
$$
We next verify that no logarithmic zero lies on the exceptional divisor. On $\{z_0\neq 0\}$, write $\sigma=\tau^{-1}$, so that $D=\{\sigma=0\}$ and $\tau\partial_\tau=-\sigma\partial_\sigma$. Since a local frame for $T_X(-\log D)$ is
$
\partial_{u_1},\ \partial_{u_2},\ \sigma\partial_\sigma,
$
equation \eqref{eq:field-u-chart-example-3fold} becomes
$$
\widetilde\xi
=
0\cdot\partial_{u_1}
+
(b-a)u_2\partial_{u_2}
+
(ka-c)(\sigma\partial_\sigma).
$$
Because $ka-c\neq 0$, this section does not vanish along $\sigma=0$. On $\{z_2\neq 0\}$, writing $\rho=\eta^{-1}$, one similarly obtains
$$
\widetilde\xi
=
(a-b)v_0\partial_{v_0}
+
(a-b)v_1\partial_{v_1}
+
(kb-c)(\rho\partial_\rho),
$$
and again there are no zeros along $D$. Therefore
\begin{equation}\label{eq:zero-locus-example-3fold}
Z(\widetilde\xi)=C\sqcup\{p\},
\qquad
Z(\widetilde\xi)\cap D=\varnothing.
\end{equation}
The component $C$ is Bott non-degenerate. Indeed, on the chart \eqref{eq:field-u-chart-example-3fold} one has $C=\{u_2=\tau=0\}$, and the transverse linear action has eigenvalues $b-a$ and $c-ka$, both nonzero by \eqref{eq:weights-example-3fold}. At the isolated point $p$, the linear part of $\widetilde\xi$ has eigenvalues $a-b$, $a-b$, and $c-kb$, again nonzero. Thus both components of \eqref{eq:zero-locus-example-3fold} are non-degenerate in the sense of the theorem.

We now compute the local contributions for $\Phi=c_3$. Since $C$ is smooth of dimension one and the transverse action is invertible, the local residue reduces to the first Chern class of $T_C$. Hence
\begin{equation}\label{eq:curve-residue-example-3fold}
\int_C \operatorname{Res}_C(c_3)
=
\int_{\mathbb P^1} c_1(T_{\mathbb P^1})
=
2.
\end{equation}
At the isolated point $p$, the local contribution is the usual Poincar\'e--Hopf index:
\begin{equation}\label{eq:point-residue-example-3fold}
\int_{\{p\}}\operatorname{Res}_p(c_3)=1.
\end{equation}
Therefore
\begin{equation}\label{eq:sum-local-example-3fold}
\sum_{Z\subset Z(\widetilde\xi)} \int_Z \operatorname{Res}_Z(c_3)=2+1=3.
\end{equation}
We turn to the global term. Set
$
h=\pi^*c_1\bigl(\mathcal O_{\mathbb P^2}(1)\bigr)
$
and
$
\xi=c_1\bigl(\mathcal O_X(1)\bigr).
$
Then the Chow ring of $X$ is generated by $h$ and $\xi$ with relations
\begin{equation}\label{eq:Chow-ring-example-3fold}
h^3=0,
\qquad
\xi^2-kh\,\xi=0.
\end{equation}
Since the exceptional section has class $D=\xi-kh$, one gets
\begin{equation}\label{eq:D-intersections-example-3fold}
h^2D=1,
\qquad
D^2=-khD,
\qquad
D^3=k^2.
\end{equation}
From the relative Euler sequence and the tangent sequence of $\pi$, one obtains
$$
c_1(T_X)=2D+(k+3)h,
\qquad
c_2(T_X)=6hD+3(k+1)h^2,
\qquad
c_3(T_X)=6h^2D.
$$
Since $D$ is smooth,
$$
c\bigl(T_X(-\log D)\bigr)=\frac{c(T_X)}{1+D},
$$
hence
\begin{equation}\label{eq:c3-log-example-3fold}
c_3\bigl(T_X(-\log D)\bigr)
=
c_3(T_X)-c_2(T_X)D+c_1(T_X)D^2-D^3.
\end{equation}
We evaluate the four terms in \eqref{eq:c3-log-example-3fold} separately. First,
$$
\int_X c_3(T_X)=\int_X 6h^2D=6.
$$
Next,
$$
\int_X c_2(T_X)D
=
\int_X \bigl(6hD+3(k+1)h^2\bigr)D
=
6\int_X hD^2+3(k+1)\int_X h^2D
=
3-3k,
$$
by \eqref{eq:D-intersections-example-3fold}. Likewise,
$$
\int_X c_1(T_X)D^2
=
\int_X \bigl(2D+(k+3)h\bigr)D^2
=
2\int_X D^3+(k+3)\int_X hD^2
=
k(k-3),
$$
and finally
$$
\int_X D^3=k^2.
$$
Substituting into \eqref{eq:c3-log-example-3fold}, we obtain
\begin{equation}\label{eq:global-example-3fold}
\int_X c_3\bigl(T_X(-\log D)\bigr)
=
6-(3-3k)+k(k-3)-k^2
=
3.
\end{equation}
Combining \eqref{eq:sum-local-example-3fold} and \eqref{eq:global-example-3fold}, we conclude that
$$
\int_X c_3\bigl(T_X(-\log D)\bigr)
=
\int_C \operatorname{Res}_C(c_3)
+
\int_{\{p\}}\operatorname{Res}_p(c_3)
=
2+1
=
3.
$$
Thus the logarithmic Bott localization formula holds in this explicit three-dimensional example, with one non-isolated smooth component and one isolated point.
\end{example}

\begin{example} \rm \label{ex:log-bott-codim-two}
Let
$$
X=\mathbb P^1_{[x_0:x_1]}\times \mathbb P^1_{[y_0:y_1]}\times \mathbb P^m_{[u_0:\dots:u_m]},
\qquad
D=\mathbb P^1\times \mathbb P^1\times H_\infty,
$$
where $m\ge 2$ and $H_\infty=\{u_0=0\}\subset \mathbb P^m$. Then $\dim X=m+2\ge 4$.
Consider the algebraic $\mathbb C^*$-action on $X$ given by
$$
\lambda\cdot\bigl([x_0:x_1],[y_0:y_1],[u_0:\dots:u_m]\bigr)
=
\bigl([x_0:\lambda x_1],[y_0:\lambda y_1],[u_0:\dots:u_m]\bigr).
$$
Let $v$ be its infinitesimal generator. Then $v$ is a global holomorphic vector field on $X$. On the affine chart
$$
t=\frac{x_1}{x_0},
\qquad
w=\frac{y_1}{y_0},
$$
one has
$$
v=t\partial_t+w\partial_w.
$$
Since the action preserves the vertical divisor $D$, the field $v$ is tangent to $D$, hence
$$
v\in H^0\bigl(X,T_X(-\log D)\bigr).
$$
The zero locus of $v$ is the fixed-point locus of the action, namely
$$
Z(v)=\bigsqcup_{\varepsilon_1,\varepsilon_2\in\{0,\infty\}}
F_{\varepsilon_1,\varepsilon_2},
$$
where
$$
F_{\varepsilon_1,\varepsilon_2}
=
\{x=\varepsilon_1,\ y=\varepsilon_2\}\times \mathbb P^m
\simeq \mathbb P^m.
$$
Thus each connected component of $Z(v)$ has positive dimension $m$ and codimension $2$ in $X$.
Fix one component $F:=F_{\varepsilon_1,\varepsilon_2}$. Since the two normal directions come from the two $\mathbb P^1$-factors, one has
$$
N_{F/X}\simeq \mathcal O_F^{\oplus 2}.
$$
The normal endomorphism $A_F$ has eigenvalues
$$
\lambda_1=
\begin{cases}
1,& \varepsilon_1=0,\\
-1,& \varepsilon_1=\infty,
\end{cases}
\qquad
\lambda_2=
\begin{cases}
1,& \varepsilon_2=0,\\
-1,& \varepsilon_2=\infty.
\end{cases}
$$
Indeed, near $0$ the induced field is $t\partial_t$, whereas near $\infty$, with
$s=1/t$, it is $-s\partial_s$.
Moreover,
$$
K_F:=\ker\bigl(T_X(-\log D)|_F\to N_{F/X}\bigr)\simeq T_{\mathbb P^m}(-\log H_\infty).
$$
Let $h=c_1(\mathcal O_{\mathbb P^m}(1))$. Since
$$
c\bigl(T_{\mathbb P^m}(-\log H_\infty)\bigr)
=
\frac{c(T_{\mathbb P^m})}{1+h}
=
\frac{(1+h)^{m+1}}{1+h}
=
(1+h)^m,
$$
it follows that
$$
c_m(K_F)=h^m,
\qquad
\int_F c_m(K_F)=1.
$$
Take $\Phi=c_{m+2}=c_{\dim X}$. Since $N_{F/X}$ is trivial of rank $2$, its formal Chern roots are $y_1=y_2=0$. If $x_1,\dots,x_m$ are the formal Chern roots of $K_F$, then
$$
\operatorname{Res}_F(c_{m+2})
=
\left[
\frac{x_1\cdots x_m(\lambda_1+y_1)(\lambda_2+y_2)}
{(\lambda_1+y_1)(\lambda_2+y_2)}
\right]_{[2m]}
=
x_1\cdots x_m
=
c_m(K_F).
$$
Hence
$$
\int_F \operatorname{Res}_F(c_{m+2})=1.
$$
Since $Z(v)$ has four connected components,
$$
\sum_F \int_F \operatorname{Res}_F(c_{m+2})=4.
$$
On the other hand,
$$
T_X(-\log D)
\simeq
p_1^*T_{\mathbb P^1}\oplus p_2^*T_{\mathbb P^1}\oplus q^*T_{\mathbb P^m}(-\log H_\infty),
$$
so
$$
c\bigl(T_X(-\log D)\bigr)
=
(1+2\eta_1)(1+2\eta_2)(1+h)^m,
$$
where $\eta_i=c_1(\mathcal O_{\mathbb P^1}(1))$ on the $i$-th $\mathbb P^1$-factor. Therefore
$$
c_{m+2}\bigl(T_X(-\log D)\bigr)=4\,\eta_1\eta_2 h^m,
$$
and thus
$$
\int_X c_{m+2}\bigl(T_X(-\log D)\bigr)=4.
$$
Therefore the global characteristic number agrees with the sum of the local logarithmic residues.
\end{example}

\begin{example}\label{subsec:fm-p2-two-points}\rm 
We briefly record a genuine modular example with nonempty boundary and non-isolated zero locus.
Let
$$
X:=(\mathbb P^2)[2]\cong \operatorname{Bl}_{\Delta}(\mathbb P^2\times \mathbb P^2),
\qquad
D:=E,
$$
where $\Delta\subset \mathbb P^2\times \mathbb P^2$ is the diagonal and $E$ is the exceptional divisor. Thus $X$ is the Fulton--MacPherson compactification of the moduli space of ordered configurations of two distinct points in $\mathbb P^2$, and
$$
X\setminus D \cong \operatorname{Conf}_2(\mathbb P^2).
$$
In particular,
$
\dim X=4.
$

Consider the one-parameter subgroup
$
\lambda\cdot [u_0:u_1:u_2]=[u_0:u_1:\lambda u_2]
$
of $\operatorname{PGL}_3$. Since the diagonal is invariant, the action lifts to $X$ and preserves $D$. Let
$
v\in H^0\bigl(X,T_X(-\log D)\bigr)
$
be the induced logarithmic vector field.
The fixed locus in $\mathbb P^2$ is
$$
L:=\{u_2=0\}\cong \mathbb P^1
\qquad\text{and}\qquad
p:=[0:0:1].
$$
Hence the zero locus of $v$ on $X$ is not isolated. Indeed, it contains positive-dimensional components outside the boundary, namely the closures of
$
L\times L,\ L\times\{p\},\ \{p\}\times L,
$
and also positive-dimensional components inside the boundary $D=E$. More precisely, over $L\subset \Delta$ one gets two fixed sections of
$$
E|_L\cong \mathbb P(T_{\mathbb P^2}|_L),
$$
corresponding to the tangent and normal directions to $L$, while over the fixed point $p$ the whole fiber
$
E_p\cong \mathbb P(T_p\mathbb P^2)\cong \mathbb P^1
$
is fixed.
Applying Theorem~\ref{thm:main} with
$
\Phi=c_4=c_{\dim X},
$
we obtain
$$
\int_X c_4\bigl(T_X(-\log D)\bigr)
=
\sum_{F\subset Z(v)} \int_F  \operatorname{Res}_F(c_4).
$$
On the other hand, by logarithmic Gauss--Bonnet,
$$
\int_X c_4\bigl(T_X(-\log D)\bigr)
=
\chi(X\setminus D)
=
\chi\bigl(\operatorname{Conf}_2(\mathbb P^2)\bigr).
$$
Since
$$
\chi\bigl(\operatorname{Conf}_2(\mathbb P^2)\bigr)
=
\chi(\mathbb P^2\times \mathbb P^2)-\chi(\Delta)
=
9-3
=
6,
$$
we conclude that
$$
\sum_{F\subset Z(v)} \int_F \operatorname{Res}_F(c_4)=6.
$$
Although this number is accessible by the classical blow-up and intersection-theoretic techniques developed for Fulton--MacPherson spaces \cite{FultonMacPherson94,Petersen17,Aluffi10}, the significance of the present computation is different. Namely, it is obtained here directly from our logarithmic localization formula on a genuine moduli space, with the crucial geometric input being the occurrence of positive-dimensional zero components, including components lying in the boundary divisor.
\end{example}

\end{document}